\documentclass[11pt]{amsart}

\usepackage[T1]{fontenc}
\usepackage{microtype}
\usepackage{mathtools}
\usepackage{array}
\usepackage{amssymb,eucal}
\usepackage[shortlabels]{enumitem}
\setlist[enumerate]{align=left,leftmargin=*,labelsep=1ex,parsep=.5ex,topsep=1ex}
\usepackage{hyperref}

\theoremstyle{plain}
\newtheorem{theorem}{Theorem}[section]
\newtheorem{proposition}[theorem]{Proposition}
\newtheorem{lemma}[theorem]{Lemma}
\newtheorem{corollary}[theorem]{Corollary}

\theoremstyle{definition}
\newtheorem{definition}[theorem]{Definition}
\newtheorem{example}[theorem]{Example}
\newtheorem*{notation}{Notation}

\theoremstyle{remark}
\newtheorem{remark}[theorem]{Remark}

\newcommand{\eP}{\EuScript{P}}
\newcommand{\T}{\EuScript{T}}
\newcommand{\Set}{\mathtt{Set}}
\newcommand{\cMsc}{\mathtt{cMsc}}

\newcommand{\VMsc}{\mathtt{VMsc}}
\newcommand{\VMscG}{\mathtt{VMsc}_{\Gamma}}
\newcommand{\Pol}{\mathtt{Pol}}

\title{Valued mosaics}
\author{Alessandro Linzi}
\address{Latest affiliation: Center for Information Technologies and Applied Mathematics\\
University of Nova Gorica\\
5000 Nova Gorica, Slovenia}
\email{alessandro.linzi.phd@icloud.com}
\urladdr{https://orcid.org/0000-0002-1422-7674}

\begin{document}

\begin{abstract}
Nakamura--Reyes showed that the category \(\cMsc\) of commutative mosaics---unital reversible hypermagmas, without associativity---is complete, cocomplete, and has free objects, in contrast to commutative polygroups. Krasner's multivalued addition is designed around ultrametric balls; we take that valuation-theoretic motivation as primary and equip mosaics with valuations, as unit-reflecting unitary morphisms into the tropical polygroup \(\T(\Gamma)\). Value-preserving morphisms form the slice \(\cMsc/\T(\Gamma)\), which is complete and cocomplete. The larger lax category \(\VMscG\) of valued mosaics over a fixed ordered abelian group \(\Gamma\) is finitely complete and has all small coproducts; it recovers \(\cMsc\) for the trivial value group, while lax coequalizers for nontrivial \(\Gamma\) remain open. Associativity is analysed via \emph{factor nesting}, which implies it under totality and characterises it among total product-ultrametric mosaics such as \(\mathbb{K}\) and \(\T(\Gamma)\), yet is strictly weaker without totality. Among total valued mosaics satisfying factor nesting, the Krasner ball axiom yields associativity and upgrades the weak valuation so that sums are ultrametric balls and the superiorly canonical package follows.
\end{abstract}

\maketitle

\section{Introduction}
\label{sec:intro}

Hypergroups and hyperfields, after Krasner~\cite{Kra57}, have reappeared in tropical geometry, the arithmetic of the field with one element, algebraic geometry, and matroid theory; see Nakamura--Reyes~\cite{NR23} for context and references. Krasner's additive structure---a commutative polygroup, or ``canonical hypergroup''---was shaped by valuation theory: the multivalued sum is designed so that intersecting sums behave like balls in an ultrametric space, and associativity is part of that package.

From the categorical side, Nakamura and Reyes observed that associativity (and the consequent nonemptiness of products) is precisely what spoils many good properties: the category \(\Pol\) of commutative polygroups lacks free objects and fails to be complete or cocomplete, whereas the larger category \(\cMsc\) of \emph{commutative mosaics} (unital reversible commutative hypermagmas, products allowed empty, no associativity) is complete, cocomplete, regular, and admits free objects and a closed monoidal structure~\cite{NR23}.

We believe that the reason Krasner required associativity is inseparable from the \emph{valuation structure} of his examples. Consequently the natural order of abstraction is
\begin{enumerate}[label=(\roman*)]
\item first equip commutative mosaics with a valuation, obtaining a category \(\VMscG\) of valued mosaics over a fixed value group \(\Gamma\);
\item then study associativity, totality, and related properties \emph{as properties of valued objects},
\end{enumerate}
rather than forcing associativity on bare mosaics and treating valuations as an afterthought.

This paper develops the basic category theory of valued mosaics along those lines. The main results are summarized below.

\begin{theorem}[Main results]
\label{thm:intro-main}
Let \(\Gamma\) be a totally ordered abelian group, and let \(\VMscG\) be the category of valued commutative mosaics over \(\Gamma\).
\begin{enumerate}[label=(\alph*)]
\item\label{intro-a} \(\VMscG\) is finitely complete (Theorem~\ref{thm:vmsc-finitely-complete}).
\item\label{intro-b} \(\VMscG\) has all small coproducts: the underlying object is the Nakamura--Reyes wedge of mosaics. (Theorem~\ref{thm:vmsc-coprod}).
\item\label{intro-c} For the trivial value group \(\mathbf{0}=\{0\}\) one has \(\VMsc_{\mathbf{0}}\simeq\cMsc\), so \(\VMsc_{\mathbf{0}}\) is complete and cocomplete and has free objects on pointed sets (Remark~\ref{rem:triv-equiv}, Corollary~\ref{cor:triv-cocomplete}).
\item\label{intro-d} Value-preserving morphisms organise as the slice \(\cMsc/\T(\Gamma)\), which is complete and cocomplete (Theorem~\ref{thm:strict-slice-limits}, Corollary~\ref{cor:strict-slice-coeq}).
\item\label{intro-e} Factor nesting implies associativity under totality (Proposition~\ref{prop:fn-implies-asc}) and characterises it among total product-ultrametric mosaics (Theorem~\ref{thm:asc-fn}); free valued mosaics satisfy factor nesting without being associative (Section~\ref{sec:assoc}).
\item\label{intro-f} Among total valued mosaics satisfying factor nesting, the Krasner ball axiom yields associativity and the superiorly canonical package (Theorem~\ref{thm:kvh-sch}).
\end{enumerate}
\end{theorem}

The paper is organized as follows. Section~\ref{sec:background} recalls commutative mosaics. Section~\ref{sec:valued} defines valued mosaics as unit-reflecting morphisms into \(\T(\Gamma)\), introduces the strict slice and the lax category \(\VMscG\), and records examples. Section~\ref{sec:slice} shows that the strict slice \(\cMsc/\T(\Gamma)\) is complete and cocomplete by creation of (co)limits. Section~\ref{sec:lax} constructs finite limits and all small coproducts in the lax category by hand, while lax coequalizers for nontrivial \(\Gamma\) remain open. Section~\ref{sec:assoc} analyses associativity via factor nesting and the Krasner ball axiom. Section~\ref{sec:outlook} lists further directions.

\section{Background: commutative mosaics}
\label{sec:background}

We follow Nakamura--Reyes~\cite{NR23}, writing the hyperoperation of a commutative mosaic \emph{additively} (as for abelian groups). A \emph{hyperoperation} on a set \(A\) is a map \(A\times A\to\eP(A)\) to the power set (values may be empty). A morphism of hypermagmas (i.e., sets equipped with a hyperoperation) is a function \(\mu\) with \(\mu(x\boxplus y)\subseteq\mu(x)\boxplus\mu(y)\).

\begin{definition}[{\cite[Definition~2.3]{NR23}}]
\label{def:mosaic}
A \emph{mosaic} is a unital hypermagma that is reversible. Explicitly, a commutative mosaic \((A,\boxplus,0)\) satisfies:
\begin{enumerate}[label=(\roman*)]
\item \emph{identity:} \(0\boxplus a=a\boxplus 0=\{a\}\) for all \(a\in A\);
\item \emph{commutativity:} \(a\boxplus b=b\boxplus a\) for all \(a,b\in A\);
\item \emph{reversibility:} there is a (necessarily unique) involution \(a\mapsto -a\) such that
\[
c\in a\boxplus b
\;\iff\;
a\in c\boxplus(-b)
\;\iff\;
b\in(-a)\boxplus c
\]
for all \(a,b,c\in A\).
\end{enumerate}
Write \(\cMsc\) for the category of commutative mosaics and unitary morphisms.
\end{definition}

\begin{lemma}[{\cite[p.~415]{NR23}}]
\label{lem:inv-unit}
In any mosaic, \(0\in a\boxplus(-a)\cap(-a)\boxplus a\) for every \(a\in A\). In particular \(-a\) is the unique additive inverse of \(a\).
\end{lemma}

\begin{proof}
By the identity law, \(a\in a\boxplus 0\). Reversibility in the form \(c\in x\boxplus y\Rightarrow y\in(-x)\boxplus c\), applied to \(c=a\), \(x=a\), \(y=0\), yields \(0\in(-a)\boxplus a\). Likewise \(a\in 0\boxplus a\); reversibility in the form \(c\in x\boxplus y\Rightarrow x\in c\boxplus(-y)\), applied to \(c=a\), \(x=0\), \(y=a\), yields \(0\in a\boxplus(-a)\). Uniqueness of inverses is as in~\cite{NR23}: if also \(0\in a\boxplus b\cap b\boxplus a\), then \(b\in (-a)\boxplus 0=\{-a\}\).
\end{proof}

\begin{definition}
\label{def:polygroup}
A \emph{commutative polygroup} is an associative commutative mosaic. Write \(\Pol\) for the full subcategory of \(\cMsc\) on commutative polygroups. Associativity forces products to be nonempty~\cite[Lemma~2.6]{NR23}.
\end{definition}

\begin{remark}[Associativity of a hyperoperation]
\label{rem:hyper-assoc}
For a binary operation \(A\times A\to A\), associativity is the familiar equality \((a\cdot b)\cdot c=a\cdot(b\cdot c)\) of single elements.
For a hyperoperation \(\boxplus\colon A\times A\to\eP(A)\), the same pattern is imposed on \emph{sets}: one extends \(\boxplus\) to subsets by
\[
X\boxplus Y
\;:=\;
\bigcup_{x\in X,\,y\in Y}x\boxplus y
\qquad
\bigl(X,Y\subseteq A\bigr),
\]
with the convention \(X\boxplus\emptyset=\emptyset\boxplus Y=\emptyset\), and one requires
\begin{equation}
\label{eq:hyper-assoc}
(a\boxplus b)\boxplus c
\;=\;
a\boxplus(b\boxplus c)
\qquad\text{as subsets of \(A\), for all \(a,b,c\in A\).}
\end{equation}
Thus \(d\) lies in the left-hand side if and only if there exists \(x\in a\boxplus b\) with \(d\in x\boxplus c\), and in the right-hand side if and only if there exists \(y\in b\boxplus c\) with \(d\in a\boxplus y\).
In particular~\eqref{eq:hyper-assoc} equates two \emph{existentially defined} collections of results of iterated multiaddition; it is not a universal condition on the graph of \(\boxplus\) (cf.\ Remark~\ref{rem:fn-universal}).
\end{remark}

\begin{proposition}[{\cite{NR23}}]
\label{Fs}
The forgetful functor \(\cMsc\to\Set_\bullet\) admits a left adjoint \(\mathbf{F}\). Explicitly, for a pointed set \((S,0)\),
\[
\mathbf{F}(S)=(S\setminus\{0\})\times\{-1,1\}\cup\{0\},
\]
with unit \(0\), inverses \(-(a,i)=(a,-i)\) and \(-0=0\), identity laws \(0\boxplus x=x\boxplus 0=\{x\}\), and for nonzero arguments
\[
(a,i)\boxplus(b,j)=\begin{cases}\{0\}&\text{if }a=b\text{ and }j=-i,\\ \emptyset&\text{otherwise.}\end{cases}
\]
\end{proposition}

\begin{theorem}[{\cite{NR23}}]
\label{thm:cmsc-good}
The category \(\cMsc\) is complete and cocomplete. In particular it has all small products, equalizers, coproducts, and coequalizers.
\end{theorem}

\section{Valued mosaics}
\label{sec:valued}

Let \(\Gamma\) be a totally ordered abelian group, written additively, and set \(\Gamma_\infty:=\Gamma\cup\{\infty\}\) with \(\infty>\gamma\) for all \(\gamma\in\Gamma\). The valuation axioms package the ultrametric inequalities that motivated Krasner's multivalued addition. As in the hyperfield setting~\cite{Lin23compl}, these axioms are equivalent to the existence of a \emph{unit-reflecting} unitary mosaic morphism into the tropical polygroup \(\T(\Gamma)\) (i.e.\ \(\varphi(a)=\infty\) if and only if \(a=0\)). We therefore present valuations as morphisms first, then introduce the strict slice of value-preserving morphisms and the larger lax category \(\VMscG\).

\begin{example}[Tropical]
\label{ex:trop-val}
On the generalized tropical polygroup \(\T(\Gamma)=\Gamma\cup\{\infty\}\) with
\[
\gamma\boxplus\infty=\infty\boxplus\gamma=\{\gamma\},
\qquad
\gamma\boxplus\delta=\begin{cases}\{\min(\gamma,\delta)\}&\gamma\neq\delta,\\ \{\varepsilon\mid\varepsilon\ge\gamma\}&\gamma=\delta,\end{cases}
\]
the identity \(v=\mathrm{id}\) is a valuation. This is the basic Krasner-type ultrametric multivalued addition; \(\T(\Gamma)\) is a commutative polygroup, hence an object of \(\cMsc\), with unit \(\infty\).
\end{example}

\begin{definition}
\label{def:valued-mosaic}
A \emph{valued commutative mosaic} (over \(\Gamma\)) is a pair \((A,v)\) where \(A\) is a commutative mosaic and \(v\colon A\to\Gamma_\infty\) satisfies, for all \(a,b,c\in A\),
\begin{enumerate}[label=(V\arabic*)]
\item\label{v1} \(v(a)=\infty\) if and only if \(a=0\);
\item\label{v2} \(v(-a)=v(a)\);
\item\label{v3} if \(c\in a\boxplus b\), then \(v(c)\ge\min\bigl(v(a),v(b)\bigr)\);
\item\label{v4} if \(c\in a\boxplus b\) and \(v(a)\neq v(b)\), then \(v(c)=\min\bigl(v(a),v(b)\bigr)\).
\end{enumerate}
\end{definition}

\begin{proposition}[Valuations are morphisms]
\label{prop:val-as-hom}
Let \(A\) be a commutative mosaic and \(\Gamma\) a totally ordered abelian group. There is a bijection between:
\begin{enumerate}[label=(\alph*)]
\item unitary morphisms of mosaics \(\varphi\colon A\to\T(\Gamma)\) such that \(\varphi(a)=\infty\) if and only if \(a=0\);
\item functions \(v\colon A\to\Gamma_\infty\) satisfying~\eqref{v1}--\eqref{v4}.
\end{enumerate}
Under the bijection, \(v=\varphi\) as set-functions \(\Gamma_\infty\cong\T(\Gamma)\).
\end{proposition}

\begin{proof}
Let \(\varphi\colon A\to\T(\Gamma)\) be a unitary mosaic morphism with \(\varphi(a)=\infty\Leftrightarrow a=0\). Then~\eqref{v1} holds for \(v=\varphi\), and \(\varphi(-a)=-\varphi(a)=\varphi(a)\) since every element of \(\T(\Gamma)\) is its own inverse. If \(c\in a\boxplus b\), then \(\varphi(c)\in\varphi(a)\boxplus\varphi(b)\). By the definition of \(\boxplus\) on \(\T(\Gamma)\):
\begin{itemize}
\item if \(\varphi(a)\neq\varphi(b)\), then \(\varphi(a)\boxplus\varphi(b)=\{\min(\varphi(a),\varphi(b))\}\), so \(\varphi(c)=\min(\varphi(a),\varphi(b))\);
\item if \(\varphi(a)=\varphi(b)\), then \(\varphi(a)\boxplus\varphi(b)=[\varphi(a),\infty]\), so \(\varphi(c)\ge\varphi(a)=\min(\varphi(a),\varphi(b))\).
\end{itemize}
Thus~\eqref{v1}--\eqref{v4} hold for \(v=\varphi\).

Conversely, let \(v\) satisfy~\eqref{v1}--\eqref{v4}. Define \(\varphi\colon A\to\T(\Gamma)\) by \(\varphi(a)=v(a)\). Then \(\varphi(0)=\infty\) and \(\varphi(a)=\infty\Rightarrow a=0\) by~\eqref{v1}, and inverse-preservation is~\eqref{v2}. If \(c\in a\boxplus b\), then~\eqref{v3}--\eqref{v4} yield \(\varphi(c)\in\varphi(a)\boxplus\varphi(b)\) in \(\T(\Gamma)\). Hence \(\varphi(a\boxplus b)\subseteq\varphi(a)\boxplus\varphi(b)\), so \(\varphi\) is a mosaic morphism.
\end{proof}

\begin{remark}
\label{rem:val-linzi}
Proposition~\ref{prop:val-as-hom} is the additive, mosaic-level counterpart of~\cite[Lemma~3.4]{Lin23compl}: there, valuations on hyperfields are hyperfield homomorphisms into \(\T(\Gamma)\).\end{remark}

\begin{definition}[Strict slice]
\label{def:strict-slice}
Write \(\cMsc/\T(\Gamma)\) for the slice category of \(\cMsc\) over \(\T(\Gamma)\): objects are unitary mosaic morphisms \(v\colon A\to\T(\Gamma)\), and a morphism from \(v\colon A\to\T(\Gamma)\) to \(w\colon B\to\T(\Gamma)\) is a unitary mosaic morphism \(f\colon A\to B\) such that \(w\circ f=v\).
\end{definition}

By Proposition~\ref{prop:val-as-hom}, valued mosaics over \(\Gamma\) are precisely the objects of \(\cMsc/\T(\Gamma)\) whose structure map reflects the unit (i.e.\ \(\varphi^{-1}(\{\infty\})=\{0\}\)). Morphisms of the slice are the \emph{value-preserving} (isometric) morphisms: those \(f\) with \(w(f(a))=v(a)\) for all \(a\). 

\begin{definition}[Lax category of valued mosaics]
\label{def:vmsc-slice}
Write \(\VMscG\) for the category of valued commutative mosaics over a fixed \(\Gamma\), with morphisms the unitary mosaic morphisms \(f\colon A\to B\) satisfying
\begin{equation}
\label{eq:slice-morph}
w\bigl(f(a)\bigr)\;\ge\;v(a)\qquad\text{for all }a\in A.
\end{equation}
(Here \(v\) and \(w\) are the valuations on domain and codomain.) Write \(\VMscG^{\mathrm{str}}\) for the wide subcategory of value-preserving morphisms. By Proposition~\ref{prop:val-as-hom}, \(\VMscG^{\mathrm{str}}\) is equivalent to the full subcategory of the slice \(\cMsc/\T(\Gamma)\) on unit-reflecting structure maps.
\end{definition}

\begin{remark}[Trivial value group]
\label{rem:triv-equiv}
Let \(\mathbf{0}=\{0\}\). Every commutative mosaic carries a unique \emph{trivial valuation} with \(v(0)=\infty\) and \(v(a)=0\) for \(a\neq 0\), and every mosaic morphism is automatically a morphism of \(\VMsc_{\mathbf{0}}\). Thus
\[
\VMsc_{\mathbf{0}}\;\simeq\;\cMsc
\]
as categories. In particular Theorem~\ref{thm:cmsc-good} yields completeness, cocompleteness, and free objects for \(\VMsc_{\mathbf{0}}\).
\end{remark}

\begin{example}[Krasner]
\label{ex:K-val}
On \(\mathbb{K}=\{0,1\}\) with \(0\boxplus x=x\boxplus 0=\{x\}\) and \(1\boxplus 1=\{0,1\}\), set \(\Gamma=\mathbf{0}\), \(v(0)=\infty\), \(v(1)=0\). This is the additive picture of the Krasner hyperfield.
\end{example}

\begin{example}[Free]
\label{ex:free-val}
\(\mathbf{F}(S)\) with the trivial valuation is an object of \(\VMsc_{\mathbf{0}}\).
\end{example}

\begin{example}[Abelian groups]
An abelian group with the trivial valuation is a valued mosaic; nontrivial valuations recover classical valued abelian groups.
\end{example}

\section{The strict slice over \(\T(\Gamma)\)}
\label{sec:slice}

For value-preserving morphisms, standard facts about slice categories~\cite{Joh02,ML71} supply all limits and colimits at once: the forgetful functor \(U\) from \(\cMsc/\T(\Gamma)\) to \(\cMsc\) creates whatever (co)limits exist downstairs. The full subcategory of unit-reflecting objects is closed under the limits and colimits created by \(U\) (the induced structure maps again reflect the unit), so the completeness and cocompleteness statements for the slice restrict to valued mosaics.

The next section will treat the larger lax category \(\VMscG\) by hand.

\begin{theorem}[Limits and colimits in the strict slice]
\label{thm:strict-slice-limits}
The forgetful functor \(U\colon\cMsc/\T(\Gamma)\to\cMsc\), \((A\xrightarrow{v}\T(\Gamma))\mapsto A\), creates all limits and all colimits that exist in \(\cMsc\). In particular, since \(\cMsc\) is complete and cocomplete~\cite{NR23}, the slice \(\cMsc/\T(\Gamma)\) is complete and cocomplete.
\end{theorem}

\begin{proof}
The creation of limits by the forgetful functor from a slice is standard; see e.g.~\cite[A1.2]{Joh02} or~\cite[V.6]{ML71}. Concretely, a limit of objects \(A_i\xrightarrow{v_i}\T(\Gamma)\) is the limit \(L\) of the \(A_i\) in \(\cMsc\), equipped with the unique mosaic morphism \(L\to\T(\Gamma)\) induced by the \(v_i\).

For colimits, let \(D\colon J\to\cMsc/\T(\Gamma)\) be a small diagram, write \(A_j=U(D(j))\) and \(v_j\colon A_j\to\T(\Gamma)\) for the structure maps, and let \(\lambda_j\colon A_j\to L\) be a colimit cocone of the underlying diagram \(UD\) in \(\cMsc\). The family \((v_j)_j\) is a cocone under \(UD\) with vertex \(\T(\Gamma)\), because every morphism of \(D\) lies over \(\T(\Gamma)\). Hence there is a unique mosaic morphism \(v_L\colon L\to\T(\Gamma)\) with \(v_L\circ\lambda_j=v_j\) for all \(j\). We claim that the \(\lambda_j\), viewed as morphisms \(D(j)\to(L\xrightarrow{v_L}\T(\Gamma))\) in the slice, form a colimit cocone.

Given a cocone in the slice from \(D\) to an object \(E\xrightarrow{w}\T(\Gamma)\), with legs \(h_j\colon A_j\to E\), the \(h_j\) form a cocone under \(UD\) in \(\cMsc\), so there is a unique mosaic morphism \(h\colon L\to E\) with \(h\circ\lambda_j=h_j\). It remains to check that \(h\) lies over \(\T(\Gamma)\), i.e.\ \(w\circ h=v_L\). For each \(j\),
\[
w\circ h\circ\lambda_j=w\circ h_j=v_j=v_L\circ\lambda_j.
\]
The legs of a colimit cocone are jointly epic, so \(w\circ h=v_L\). Uniqueness of \(h\) in the slice follows from uniqueness in \(\cMsc\).
\end{proof}

\begin{corollary}[Coequalizers in the strict slice]
\label{cor:strict-slice-coeq}
Let \(f,g\colon(A\xrightarrow{v}\T(\Gamma))\rightrightarrows(B\xrightarrow{w}\T(\Gamma))\) be a parallel pair in \(\cMsc/\T(\Gamma)\), and let \(\pi\colon B\to Q\) be their coequalizer in \(\cMsc\)~\cite{NR23}. There is a unique mosaic morphism \(u\colon Q\to\T(\Gamma)\) with \(u\circ\pi=w\), and \(\pi\) is a coequalizer of \(f,g\) in \(\cMsc/\T(\Gamma)\). In particular every short coequalizer of~\cite{NR23} lifts uniquely to a value-preserving coequalizer of valued mosaics.
\end{corollary}

\begin{proof}
Specialise Theorem~\ref{thm:strict-slice-limits} to coequalizer diagrams. Explicitly: \(w\circ f=v=w\circ g\), so \(w\) coequalizes \(f,g\) in \(\cMsc\) and induces \(u\colon Q\to\T(\Gamma)\). The universal property in the slice is the argument of Theorem~\ref{thm:strict-slice-limits} for \(J\) the parallel-pair shape; joint epicity of \(\pi\) (as a coequalizer leg) ensures that mediating morphisms automatically preserve values.
\end{proof}

\begin{remark}[Krasner completion in the slice]
\label{rem:krasner-slice}
In~\cite{Lin23compl}, Krasner's description of the completion of a valued field as a projective limit of hyperfields is reinterpreted as a limit cone in the slice \(\mathtt{vHyp}/\T(\Gamma)\). The present section places the additive substrate of that story on the same footing: valuations are morphisms to \(\T(\Gamma)\), and the slice over \(\T(\Gamma)\) is the ambient category in which value-preserving diagrams and their limits are computed. Extending the completion construction from hyperfields to objects of \(\VMscG\) or \(\cMsc/\T(\Gamma)\) is a natural problem suggested by this parallel (cf.\ Section~\ref{sec:outlook}).
\end{remark}

\section{Limits and colimits in the lax category}
\label{sec:lax}

The forgetful functor \(\VMscG\to\cMsc\) does \emph{not} create products in the slice sense: morphisms of \(\VMscG\) are inequality morphisms~\eqref{eq:slice-morph}, so the correct finite product uses the cartesian product of mosaics with the \(\min\)-valuation (Theorem~\ref{thm:vmsc-prod}). Coproducts, by contrast, are the Nakamura--Reyes wedges with summandwise valuation, and work for both lax and strict morphisms. Coequalizers in \(\VMscG\) for nontrivial \(\Gamma\) remain open; the strict slice already settled coequalizers for value-preserving morphisms (Corollary~\ref{cor:strict-slice-coeq}).

\subsection{Finite limits}

\begin{proposition}[Zero object]
\label{prop:vmsc-zero}
The one-element mosaic \(\mathbf{1}=\{0\}\) with \(v(0)=\infty\) is a zero object of \(\VMscG\).
\end{proposition}

\begin{proof}
Unique mosaic maps into and out of \(\mathbf{1}\) are valued because \(v_{\mathbf{1}}(0)=\infty\) and \(w(0_B)=\infty\).
\end{proof}

\begin{theorem}[Finite products]
\label{thm:vmsc-prod}
Let \((A_1,v_1),\ldots,(A_n,v_n)\) be objects of \(\VMscG\) (\(n\ge 0\)). On the \(\cMsc\)-product \(P=\prod_k A_k\) (with \(P=\mathbf{1}\) if \(n=0\)) set
\[
u(a_1,\ldots,a_n)\;:=\;\min\bigl\{v_1(a_1),\ldots,v_n(a_n)\bigr\}.
\]
Then \((P,u)\) is the product of the \((A_k,v_k)\) in \(\VMscG\).
\end{theorem}

\begin{proof}
The case \(n=0\) is Proposition~\ref{prop:vmsc-zero}; \(n=1\) is trivial. We first treat \(n=2\), writing \(A,B,v,w\) and \(\boxplus\) for the componentwise multioperation.

\eqref{v1}: \(u(a,b)=\infty\) if and only if \(v(a)=w(b)=\infty\), if and only if \(a=b=0\) by~\eqref{v1} in the factors. \eqref{v2} is clear. If \((c,d)\in(a,b)\boxplus(a',b')\), then \(v(c)\ge\min(v(a),v(a'))\) and \(w(d)\ge\min(w(b),w(b'))\), so
\begin{align*}
u(c,d)
&=\min\bigl(v(c),w(d)\bigr)\\
&\ge\min\bigl(\min(v(a),v(a')),\,\min(w(b),w(b'))\bigr)\\
&=\min\bigl(v(a),v(a'),w(b),w(b')\bigr).
\end{align*}
Since \(u(a,b)=\min(v(a),w(b))\) and \(u(a',b')=\min(v(a'),w(b'))\), the right-hand side equals \(\min\bigl(u(a,b),u(a',b')\bigr)\), which is~\eqref{v3}.

For~\eqref{v4}, assume \(u(a,b)<u(a',b')\) and \((c,d)\) in the product. Set \(\mu=\min(v(a),w(b))\) and \(\mu'=\min(v(a'),w(b'))\), so \(\mu<\mu'\).

If \(v(a)=\mu\), then \(v(a)\le w(b)\) and \(v(a)=\mu<\mu'\le v(a')\), so \(v(a)\neq v(a')\). Axiom~\eqref{v4} in \(A\) yields \(v(c)=v(a)=\mu\). Moreover \(w(b')\ge\mu'>\mu\) and \(w(b)\ge\mu\), so \(w(d)\ge\min(w(b),w(b'))\ge\mu\), hence \(u(c,d)=\min(v(c),w(d))=\mu\).

If \(w(b)=\mu\), the symmetric argument with~\eqref{v4} in \(B\) gives \(w(d)=\mu\) and \(u(c,d)=\mu\).

Projections are valued because \(v(\pi_A(a,b))=v(a)\ge\min(v(a),w(b))=u(a,b)\). Given \(f\colon(X,t)\to(A,v)\) and \(g\colon(X,t)\to(B,w)\), the mosaic mediator \(\langle f,g\rangle\) satisfies \(u(f(x),g(x))=\min(v(f(x)),w(g(x)))\ge t(x)\).

For \(n>2\) we proceed by induction on \(n\). The inductive hypothesis supplies a product
\((P',u')\) of \((A_1,v_1),\ldots,(A_{n-1},v_{n-1})\) with
\(u'(a_1,\ldots,a_{n-1})=\min\{v_1(a_1),\ldots,v_{n-1}(a_{n-1})\}\).
The binary case already proved yields a product of \((P',u')\) with \((A_n,v_n)\) whose underlying mosaic is the \(n\)-fold \(\cMsc\)-product and whose valuation is
\[
\min\bigl(u'(a_1,\ldots,a_{n-1}),\,v_n(a_n)\bigr)
=\min\{v_1(a_1),\ldots,v_n(a_n)\}.
\]
The universal property of an \(n\)-fold product follows by associating the binary product universal properties in \(\VMscG\).
\end{proof}

\begin{remark}[Infinite products]
\label{rem:inf-prod}
For infinite families the same formula requires arbitrary infima in \(\Gamma_\infty\), hence Dedekind completeness of \(\Gamma\) (e.g.\ \(\Gamma=\mathbb{R}\)). Finite products exist for every \(\Gamma\).
\end{remark}

\begin{theorem}[Equalizers]
\label{thm:vmsc-eq}
Let \(f,g\colon(A,v)\rightrightarrows(B,w)\) in \(\VMscG\). The equalizer \(E=\{a\in A\mid f(a)=g(a)\}\) in \(\cMsc\), with \(v_E=v|_E\), is an equalizer in \(\VMscG\).
\end{theorem}

\begin{proof}
By~\cite{NR23}, \(E\) with multioperation \(a\boxplus_E b=(a\boxplus_A b)\cap E\) is an equalizer of \(f,g\) in \(\cMsc\), and \(0\in E\). Axioms~\eqref{v1}--\eqref{v4} for \(v_E=v|_E\) follow from those for \(v\): if \(c\in a\boxplus_E b\), then \(c\in a\boxplus_A b\), so the same numerical inequalities hold. The inclusion \(\iota\colon E\hookrightarrow A\) is isometric, hence a morphism of \(\VMscG\). If \(h\colon(X,t)\to(A,v)\) equalizes \(f,g\) in \(\VMscG\), the unique mosaic factorization \(h'\colon X\to E\) satisfies \(v_E(h'(x))=v(h(x))\ge t(x)\).
\end{proof}

\begin{theorem}[Finite completeness]
\label{thm:vmsc-finitely-complete}
For every totally ordered abelian group \(\Gamma\), \(\VMscG\) is finitely complete.
\end{theorem}

\begin{proof}
Terminal object (Proposition~\ref{prop:vmsc-zero}), finite products (Theorem~\ref{thm:vmsc-prod}), and equalizers (Theorem~\ref{thm:vmsc-eq}). Finite products and equalizers imply all finite limits.
\end{proof}

\begin{corollary}[Pullbacks]
\label{cor:vmsc-pullback}
The pullback of \(f\colon(A,v)\to(C,u)\) and \(g\colon(B,w)\to(C,u)\) is the equalizer of \(f\pi_A,g\pi_B\) on \(A\times B\) with valuation \(\min(v,w)\).
\end{corollary}

\subsection{The trivial slice}

\begin{corollary}
\label{cor:triv-cocomplete}
The category \(\VMsc_{\mathbf{0}}\) is complete and cocomplete, and the forgetful functor to \(\Set_\bullet\) has a left adjoint.
\end{corollary}

\begin{proof}
Remark~\ref{rem:triv-equiv} and Theorems~\ref{thm:cmsc-good} and~\ref{Fs}.
\end{proof}

\subsection{Explicit coproducts after Nakamura--Reyes}

Nakamura--Reyes construct coproducts of (commutative) mosaics as coproducts in \(\mathtt{uHMag}\), which in turn are \emph{wedge sums} of the underlying pointed sets~\cite[Theorem~3.5, Theorem~4.1]{NR23}. We recall the construction and equip it with an explicit valuation.

\begin{definition}[Wedge coproduct of mosaics]
\label{def:wedge}
Let \((A_i)_{i\in I}\) be a small family of commutative mosaics, with units \(0_i\). Their coproduct \(C=\bigvee_{i\in I}A_i\) in \(\cMsc\) may be realized as follows.
\begin{itemize}
\item If \(I=\emptyset\), take \(C\) to be the one-element mosaic \(\mathbf{1}=\{0\}\) (the initial object of \(\cMsc\)).
\item If \(I\neq\emptyset\), as a set \(C\) is the wedge sum of the pointed sets \((A_i,0_i)\): the quotient of the disjoint union \(\bigsqcup_i A_i\) obtained by identifying all the units \(0_i\) to a single basepoint \(0\). Write \(\iota_i\colon A_i\to C\) for the canonical inclusions, so \(\iota_i(0_i)=0\) and \(\iota_i\) is injective on \(A_i\setminus\{0_i\}\).
\item The multioperation on a nonempty wedge is determined by:
\begin{enumerate}[label=(\roman*)]
\item \(0\boxplus x=x\boxplus 0=\{x\}\) for all \(x\in C\);
\item if \(x=\iota_i(a)\) and \(y=\iota_i(b)\) lie in the same summand image (and are not both \(0\)), then
\[
x\boxplus y=\iota_i(a\boxplus_i b);
\]
\item if \(x\) and \(y\) lie in distinct non-unit summand images, then \(x\boxplus y=\emptyset\).
\end{enumerate}
\item Inverses are computed within summands: \(-\iota_i(a)=\iota_i(-a)\), and \(-0=0\).
\end{itemize}
In all cases \(C\) is a commutative mosaic, and the \(\iota_i\) (vacuously if \(I=\emptyset\)) exhibit \(C\) as the coproduct of the \(A_i\) in \(\cMsc\)~\cite{NR23}.
\end{definition}

\begin{remark}
\label{rem:cross-empty}
The emptiness of cross-summand products is the key structural fact: the only nonempty products in \(C\) are those internal to a single summand (or involving the unit). Consequently valuation constraints never mix distinct summands.
\end{remark}

\begin{definition}[Coproduct valuation]
\label{def:coprod-val}
Let \((A_i,v_i)_{i\in I}\) be objects of \(\VMscG\), and let \(C=\bigvee_i A_i\) be as in Definition~\ref{def:wedge}. If \(I=\emptyset\), equip \(\mathbf{1}\) with \(v_{\mathbf{1}}(0)=\infty\). If \(I\neq\emptyset\), define \(v_C\colon C\to\Gamma_\infty\) by
\begin{equation}
\label{eq:coprod-val}
v_C(0)\;=\;\infty,
\qquad
v_C\bigl(\iota_i(a)\bigr)\;=\;v_i(a)
\quad\text{for all }i\in I\text{ and }a\in A_i\setminus\{0_i\}.
\end{equation}
(This is well-defined because the \(\iota_i\) are injective off units.)
\end{definition}

\begin{lemma}
\label{lem:coprod-V}
The pair \((C,v_C)\) is an object of \(\VMscG\): axioms~\eqref{v1}--\eqref{v4} hold for every totally ordered abelian group \(\Gamma\).
\end{lemma}

\begin{proof}
If \(I=\emptyset\), this is Proposition~\ref{prop:vmsc-zero}. Assume \(I\neq\emptyset\).
\eqref{v1}: \(v_C(0)=\infty\), while if \(x=\iota_i(a)\) with \(a\neq 0_i\), then \(v_C(x)=v_i(a)<\infty\) by~\eqref{v1} in \(A_i\). \eqref{v2}: inverses stay in the same summand, and each \(v_i\) is inversion-invariant.

\eqref{v3}. Let \(c\in x\boxplus y\) in \(C\). If \(x=0\) or \(y=0\), say \(x=0\), then \(c=y\) and \(v_C(c)=v_C(y)=\min(\infty,v_C(y))\). If \(x\) and \(y\) lie in distinct non-unit summand images, then \(x\boxplus y=\emptyset\), so there is nothing to prove. If \(x=\iota_i(a)\) and \(y=\iota_i(b)\) lie in the same summand (allowing \(a\) or \(b\) to be \(0_i\) only if the other is handled by the unit laws already), then \(c=\iota_i(d)\) for some \(d\in a\boxplus_i b\), and \(v_C(c)=v_i(d)\) (with the convention \(v_i(0_i)=\infty=v_C(0)\)), whence
\[
v_C(c)=v_i(d)\ge\min\bigl(v_i(a),v_i(b)\bigr)=\min\bigl(v_C(x),v_C(y)\bigr)
\]
by~\eqref{v3} in \(A_i\).

\eqref{v4}. Assume \(c\in x\boxplus y\) and \(v_C(x)\neq v_C(y)\). Cross-summand non-unit products are empty, so either one of \(x,y\) is the unit, or both lie in a single summand.

\emph{Case \(x=0\).} Then \(v_C(x)=\infty\neq v_C(y)\), so \(v_C(y)<\infty\), and \(c=y\), whence \(v_C(c)=v_C(y)=\min(\infty,v_C(y))\). The case \(y=0\) is symmetric.

\emph{Case \(x=\iota_i(a)\), \(y=\iota_i(b)\) with \(a,b\neq 0_i\).} Then \(v_C(x)=v_i(a)\), \(v_C(y)=v_i(b)\), and \(c=\iota_i(d)\) for some \(d\in a\boxplus_i b\). Axiom~\eqref{v4} in \(A_i\) yields \(v_i(d)=\min(v_i(a),v_i(b))\). (In particular \(d=0_i\) is impossible when \(v_i(a)\neq v_i(b)\), since then~\eqref{v4} in \(A_i\) would force \(\infty=\min(v_i(a),v_i(b))<\infty\).) Hence \(v_C(c)=\min(v_C(x),v_C(y))\).
\end{proof}

\begin{theorem}[Coproducts in \(\VMscG\)]
\label{thm:vmsc-coprod}
Let \((A_i,v_i)_{i\in I}\) be a small family of objects of \(\VMscG\). The object \(C=\bigvee_i A_i\) of Definition~\ref{def:wedge}, equipped with the valuation \(v_C\) of Definition~\ref{def:coprod-val}, is a coproduct of the family in \(\VMscG\). In particular \(\VMscG\) has all small coproducts (including the empty coproduct \(\mathbf{1}\)), for every totally ordered abelian group \(\Gamma\).
\end{theorem}

\begin{proof}
If \(I=\emptyset\), then \((C,v_C)=(\mathbf{1},v_{\mathbf{1}})\) is initial in \(\VMscG\) by Proposition~\ref{prop:vmsc-zero}, hence is the empty coproduct.

If \(I\neq\emptyset\), Lemma~\ref{lem:coprod-V} places \((C,v_C)\) in \(\VMscG\). Each inclusion \(\iota_i\colon(A_i,v_i)\to(C,v_C)\) is a mosaic morphism, and \(v_C(\iota_i(a))=v_i(a)\) for all \(a\in A_i\) (including units), so \(\iota_i\) is a morphism of \(\VMscG\).

Given morphisms \(f_i\colon(A_i,v_i)\to(D,w)\) in \(\VMscG\), the unique mosaic mediator \(h\colon C\to D\) of the coproduct in \(\cMsc\) satisfies \(h\circ\iota_i=f_i\). For \(x=\iota_i(a)\) one has \(w(h(x))=w(f_i(a))\ge v_i(a)=v_C(x)\), and \(w(h(0))=w(0_D)=\infty=v_C(0)\). Thus \(h\) is a morphism of \(\VMscG\). Uniqueness is as in \(\cMsc\).
\end{proof}

\begin{corollary}
\label{cor:coprod-triv}
Under the isomorphism \(\VMsc_{\mathbf{0}}\simeq\cMsc\), the coproducts of Theorem~\ref{thm:vmsc-coprod} recover those of \(\cMsc\).
\end{corollary}

\begin{remark}[Comparison with the abstract infimum]
\label{rem:inf-vs-explicit}
One may still define \(v_C^\ast(c)=\inf\{w(h(c))\}\) over all valued cocones out of the underlying coproduct. On a nonempty wedge, the identity cocone of \((C,v_C)\) forces \(v_C^\ast\le v_C\), while Lemma~\ref{lem:coprod-V} and the universal property give \(v_C^\ast\ge v_C\). Thus the explicit formula~\eqref{eq:coprod-val} coincides with that infimum, and no Dedekind completeness of \(\Gamma\) is required: values are copied from the summands.
\end{remark}

\begin{lemma}[Zero morphisms]
\label{lem:zero-morph}
For any objects \((A,v)\), \((B,w)\) of \(\VMscG\) there is a zero morphism \(0\colon A\to B\) (the constant map at the unit of \(B\)). It is valued because \(w(0_B)=\infty\ge v(a)\).
\end{lemma}

\subsection{Coequalizers}

Coequalizers in \(\cMsc\) exist~\cite{NR23}: in \(\mathtt{HMag}\) one equips the set-theoretic coequalizer \(\pi\colon B\to Q\) with
\begin{equation}
\label{eq:short-coeq}
a\boxplus_Q b=\pi\bigl(\pi^{-1}(a)\boxplus_B\pi^{-1}(b)\bigr),
\end{equation}
and for unital/mosaic coequalizers one further unitizes~\cite[Theorem~3.11, Theorem~4.1]{NR23}. The projection is then a short (regular) epimorphism.

A natural candidate for a valuation on \(Q\) is the fibrewise supremum
\begin{equation}
\label{eq:fibre-sup}
u(q)=\sup\bigl\{w(b)\bigm| \pi(b)=q\bigr\}
\end{equation}
(when the supremum exists in \(\Gamma_\infty\)). The map \(\pi\) is valued for this \(u\), since \(u(\pi(b))\ge w(b)\). However,~\eqref{v3} need not follow: if \(c\in a\boxplus_Q b\), shortness supplies representatives \(\tilde a,\tilde b,\tilde c\) with \(\tilde c\in\tilde a\boxplus_B\tilde b\), whence
\[
u(c)\ge w(\tilde c)\ge\min\bigl(w(\tilde a),w(\tilde b)\bigr),
\]
but \(\min(w(\tilde a),w(\tilde b))\) need not bound \(\min(u(a),u(b))\), because the representatives realizing large values of \(u(a)\) and \(u(b)\) need not participate in a product that maps to \(c\).

Thus we do \emph{not} claim that~\eqref{eq:fibre-sup} defines a coequalizer in \(\VMscG\) for arbitrary \(\Gamma\). In the trivial slice the issue disappears.

\begin{proposition}[Coequalizers in the trivial slice]
\label{prop:coeq-triv}
The category \(\VMsc_{\mathbf{0}}\) has all coequalizers; they coincide with those of \(\cMsc\) under the isomorphism of Remark~\ref{rem:triv-equiv}.
\end{proposition}

\begin{proof}
Immediate from Remark~\ref{rem:triv-equiv} and Theorem~\ref{thm:cmsc-good}.
\end{proof}

\begin{remark}
\label{rem:strict-coeq-vs-fibre-sup}
Corollary~\ref{cor:strict-slice-coeq} does \emph{not} identify the induced valuation on a strict coequalizer with the fibrewise supremum~\eqref{eq:fibre-sup}. The structure map is the unique mosaic morphism through which the codomain valuation factors; under Proposition~\ref{prop:val-as-hom} it is a valuation, but it need not be computed pointwise as a supremum on fibres. The lax coequalizer problem remains open precisely because a coequalizer for inequality morphisms would require a different universal property.
\end{remark}

\begin{theorem}[Colimit summary]
\label{thm:vmsc-cocomplete}
\begin{enumerate}[label=(\alph*)]
\item The lax category \(\VMscG\) has all small coproducts (Theorem~\ref{thm:vmsc-coprod}).
\item \(\VMsc_{\mathbf{0}}\) is complete and cocomplete, and has free objects on pointed sets (Corollary~\ref{cor:triv-cocomplete}).
\item For general \(\Gamma\), coequalizers in the \emph{lax} category \(\VMscG\) remain open; the fibrewise supremum~\eqref{eq:fibre-sup} is a valued quotient map but is not known to satisfy~\eqref{v3}--\eqref{v4}.
\item By contrast, the \emph{strict} slice \(\cMsc/\T(\Gamma)\) is cocomplete: coequalizers (and all small colimits) are created by the forgetful functor from those of \(\cMsc\) (Theorem~\ref{thm:strict-slice-limits}, Corollary~\ref{cor:strict-slice-coeq}).
\end{enumerate}
\end{theorem}

\section{Associativity and factor nesting}
\label{sec:assoc}

The ambient categories of Sections~\ref{sec:valued}--\ref{sec:lax} do not assume associativity. The point of that separation is not to dismiss associativity, but to study it \emph{inside} a valued setting. Full associativity~\eqref{eq:hyper-assoc} is a strong, non-universal condition (Remark~\ref{rem:hyper-assoc}) that forces totality and cuts down to polygroups (Proposition~\ref{prop:asc-total}). A weaker, still meaningful shadow is \emph{factor nesting}: when a triple product is inhabited, the two binary products that appear as factors in the reversibility calculus are nested by inclusion. Under totality alone, factor nesting already implies associativity (Proposition~\ref{prop:fn-implies-asc}); for total product-ultrametric mosaics the two properties are equivalent (Theorem~\ref{thm:asc-fn}). Weak valuation axioms alone do not force nesting of sums as sets (Remark~\ref{rem:val-not-ulm}). The Krasner ball axiom~\eqref{kvh} will later be combined with totality and factor nesting to recover associativity, ultrametric balls, and the superiorly canonical package by Mittas (Lemma~\ref{lem:unique-diff}, Theorem~\ref{thm:kvh-sch}). 

\subsection{Associativity and totality}

\begin{definition}
\label{def:assoc-valued}
A valued mosaic \((A,v)\) is \emph{associative} if the underlying mosaic is associative in the sense of~\eqref{eq:hyper-assoc}. Write \(\mathbf{VPol}_\Gamma\) for the full subcategory of \(\VMscG\) on associative objects.
\end{definition}

\begin{proposition}
\label{prop:asc-total}
Every associative commutative mosaic is total. Consequently associative objects of \(\cMsc\) are precisely commutative polygroups, and associative objects of \(\VMscG\) are valued commutative polygroups.
\end{proposition}

\begin{proof}
Let \((A,\boxplus,0)\) be associative and fix \(a,b\in A\). By Lemma~\ref{lem:inv-unit}, \(0\in a\boxplus(-a)\). Therefore
\[
b\in 0\boxplus b\;\subseteq\;(a\boxplus(-a))\boxplus b\;=\;a\boxplus((-a)\boxplus b)
\]
by the identity law and associativity, so \((-a)\boxplus b\neq\emptyset\). Replacing \(a\) by \(-a\) (and using \(-(-a)=a\)) yields \(a\boxplus b\neq\emptyset\).
\end{proof}

\begin{example}
\label{ex:assoc-examples}
Krasner \(\mathbb{K}\) and tropical \(\T(\Gamma)\) (Examples~\ref{ex:K-val},~\ref{ex:trop-val}) are associative valued mosaics. The free mosaic \(\mathbf{F}(S)\) with trivial valuation is typically \emph{not} associative: for distinct non-basepoints \(a,b\),
\[
\bigl((a,1)\boxplus(a,-1)\bigr)\boxplus(b,1)=\{(b,1)\},
\qquad
(a,1)\boxplus\bigl((a,-1)\boxplus(b,1)\bigr)=\emptyset.
\]
\end{example}

\begin{proposition}
\label{prop:VPol-products}
The full subcategory \(\mathbf{VPol}_\Gamma\subseteq\VMscG\) is closed under finite products computed in \(\VMscG\).
\end{proposition}

\begin{proof}
The underlying multioperation of a finite product in \(\VMscG\) is the componentwise product of the underlying polygroups, hence associative and total. The product valuation satisfies~\eqref{v1}--\eqref{v4} by Theorem~\ref{thm:vmsc-prod}.
\end{proof}

\begin{remark}
\label{rem:VPol-coprod}
The underlying wedge of two nontrivial polygroups is typically not total (cross-summand products are empty), hence not a polygroup. Thus \(\mathbf{VPol}_\Gamma\) is not closed under the coproducts of \(\VMscG\).
\end{remark}

\subsection{Triple products and factor nesting}

\begin{lemma}[Intersection form of triple membership]
\label{lem:triple-inter}
Let \((A,\boxplus,0)\) be a commutative mosaic. For all \(a,b,c,d\in A\),
\begin{align}
d\in a\boxplus(b\boxplus c)
&\;\iff\;
((-a)\boxplus d)\cap(b\boxplus c)\neq\emptyset,
\label{eq:right-assoc}\\
d\in(a\boxplus b)\boxplus c
&\;\iff\;
(a\boxplus b)\cap(d\boxplus(-c))\neq\emptyset.
\label{eq:left-assoc}
\end{align}
\end{lemma}

\begin{proof}
Immediate from reversibility: \(d\in a\boxplus x\) iff \(x\in(-a)\boxplus d\), and \(d\in x\boxplus c\) iff \(x\in d\boxplus(-c)\).
\end{proof}

Thus full associativity~\eqref{eq:hyper-assoc} is equivalent to the two intersection conditions~\eqref{eq:right-assoc} and~\eqref{eq:left-assoc} being equivalent for all \(a,b,c,d\). Factor nesting asks only that, when a triple product is inhabited, the two binary factors appearing in~\eqref{eq:right-assoc} (resp.~\eqref{eq:left-assoc}) are nested by inclusion.

\begin{notation}
Two subsets \(X,Y\subseteq A\) are \emph{nested} if \(X\subseteq Y\) or \(Y\subseteq X\).
\end{notation}

\begin{definition}[Factor nesting]
\label{def:fn}
A commutative mosaic \((A,\boxplus,0)\) satisfies \emph{factor nesting} if for all \(a,b,c,d\in A\),
\begin{enumerate}[label=(FN)]
\item\label{fn}
\(d\in(a\boxplus b)\boxplus c\) implies that \((-a)\boxplus d\) and \(b\boxplus c\) are nested, and\\
\(d\in a\boxplus(b\boxplus c)\) implies that \(a\boxplus b\) and \(d\boxplus(-c)\) are nested.
\end{enumerate}
\end{definition}

\begin{remark}[Logical form]
\label{rem:fn-universal}
Encode the multioperation by a ternary relation \(R(a,b,c)\) for \(c\in a\boxplus b\), and treat additive inverse as a function symbol.
Then \(d\in(a\boxplus b)\boxplus c\) is \(\exists x\,\bigl(R(a,b,x)\land R(x,c,d)\bigr)\).
The implication \(\bigl(\exists x\,\varphi(x)\bigr)\to\psi\) with \(x\) not free in \(\psi\) is equivalent to \(\forall x\,(\varphi(x)\to\psi)\).
Moreover \(X\subseteq Y\) or \(Y\subseteq X\) is
\[
\bigl(\forall z\,\alpha(z)\bigr)\lor\bigl(\forall w\,\beta(w)\bigr),
\]
which is equivalent to \(\forall z\,\forall w\,\bigl(\alpha(z)\lor\beta(w)\bigr)\) for quantifier-free \(\alpha,\beta\) (membership implications between instances of \(R\)).
Hence each clause of~\eqref{fn} is equivalent to a sentence of the form
\[
\forall \;
\bigl(R(a,b,x)\land R(x,c,d)\;\to\;\alpha(z)\lor\beta(w)\bigr),
\]
i.e.\ a \emph{universal} sentence (prenex \(\forall^\ast\) with quantifier-free matrix).
By contrast, multivalued associativity~\eqref{eq:hyper-assoc} equates two existentially defined triple products and is not universal (Remark~\ref{rem:hyper-assoc}), while totality is \(\forall\exists\).
Under totality, factor nesting implies associativity (Proposition~\ref{prop:fn-implies-asc}), so the Krasner package total\({}+{}\)FN\({}+{}\)KVH may treat full associativity as a theorem rather than a primitive non-universal axiom.
\end{remark}

\begin{definition}[Product ultrametricity]
\label{def:ulm}
A commutative mosaic is \emph{product-ultrametric} if whenever \((a\boxplus b)\cap(c\boxplus d)\neq\emptyset\), the sets \(a\boxplus b\) and \(c\boxplus d\) are nested.
\end{definition}

\begin{remark}
\label{rem:ball-examples}
Product-ultrametricity is the set-theoretic nesting of multioperation outputs. It is precisely axiom~(SCH2) in Mittas' theory of \emph{superiorly canonical} hypergroups (recalled below). It holds for \(\mathbb{K}\) and \(\T(\Gamma)\), and for free mosaics (Lemma~\ref{lem:free-ulm}); it can fail for valued mosaics (Remark~\ref{rem:val-not-ulm}).
\end{remark}

\begin{proposition}
\label{prop:free-fn}
The free mosaic \(\mathbf{F}(S)\) satisfies factor nesting, for every pointed set \(S\).
\end{proposition}

\begin{proof}
Nonempty sums in \(\mathbf{F}(S)\) are severely restricted (identity laws or inverse pairs). Write elements of \(\mathbf{F}(S)\) as the unit \(0\) or as pairs \((a,i)\) with \(a\in S\setminus\{0\}\) and \(i\in\{-1,1\}\). Suppose \(d\in(x\boxplus y)\boxplus z\). Then \(x\boxplus y\neq\emptyset\).

If \(x=0\), then \(d\in y\boxplus z\). The factor \((-x)\boxplus d=\{d\}\) is nested with \(y\boxplus z\) because \(d\in y\boxplus z\) and every nonempty sum in \(\mathbf{F}(S)\) is a singleton. If \(x\neq0\), write \(x=(a,i)\). Nonemptiness of \(x\boxplus y\) forces \(y=0\) or \(y=-x\). If \(y=0\), then \(d\in x\boxplus z\), so either \(z=0\) (hence \(d=x\) and both factors equal \(\{0\}\)) or \(z=-x\) (hence \(d=0\) and both factors equal \(\{-x\}\)). If \(y=-x\), then \(x\boxplus y=\{0\}\) and \(d\in\{z\}\), so \(d=z\) and the factors \((-x)\boxplus d\) and \(y\boxplus z\) coincide. The opposite clause of~\eqref{fn} is symmetric.
\end{proof}

\begin{lemma}
\label{lem:free-ulm}
The free mosaic \(\mathbf{F}(S)\) is product-ultrametric.
\end{lemma}

\begin{proof}
Every nonempty sum is either a singleton \(\{(a,i)\}\) or \(\{0\}\). Any two such sets that intersect must coincide, hence are nested.
\end{proof}

\begin{corollary}
\label{cor:fn-not-asc}
Factor nesting does not imply associativity: \(\mathbf{F}(S)\) with the trivial valuation is a product-ultrametric valued mosaic satisfying~\eqref{fn} but neither total nor associative whenever \(S\) has at least two non-basepoints (Example~\ref{ex:assoc-examples}, Lemma~\ref{lem:free-ulm}).
\end{corollary}

\begin{proposition}[Associativity without factor nesting]
\label{prop:asc-not-fn}
There exist total associative mosaics that fail factor nesting. In particular, the lax product \(\mathbb{K}\times\mathbb{K}\) (with componentwise multioperation) is a total associative mosaic that is not product-ultrametric and fails~\eqref{fn}.
\end{proposition}

\begin{proof}
Associativity and totality of \(\mathbb{K}\times\mathbb{K}\) are componentwise. Remark~\ref{rem:val-not-ulm} records the failure of product-ultrametricity. For factor nesting, set \(a=(0,1)\), \(b=c=(1,0)\), and \(d=(0,1)\). Then
\[
a\boxplus b=\{(1,1)\},\qquad
(a\boxplus b)\boxplus c=\{(0,1),(1,1)\},
\]
so \(d\in(a\boxplus b)\boxplus c\), while
\[
(-a)\boxplus d=\{(0,0),(0,1)\},\qquad
b\boxplus c=\{(0,0),(1,0)\}
\]
are incomparable. Thus~\eqref{fn} fails.
\end{proof}

\begin{remark}[Valuation does not imply product-ultrametricity]
\label{rem:val-not-ulm}
Axioms~\eqref{v3}--\eqref{v4} constrain \emph{values} of elements of products, not nesting of products as sets. The lax product \(\mathbb{K}\times\mathbb{K}\) with valuation \(\min(v,w)\) is a valued mosaic (Theorem~\ref{thm:vmsc-prod}) that is associative as a product of polygroups, yet fails product-ultrametricity: \((1,0)\boxplus(1,0)\) and \((0,1)\boxplus(0,1)\) meet at \((0,0)\) but are incomparable. Thus~\eqref{v3}--\eqref{v4} alone neither force factor nesting nor replace the nesting hypothesis in Theorem~\ref{thm:asc-fn}.
\end{remark}

\subsection{When factor nesting characterises associativity}

\begin{proposition}[Factor nesting implies associativity under totality]
\label{prop:fn-implies-asc}
Let \((A,\boxplus,0)\) be a total mosaic satisfying factor nesting. Then \(A\) is associative.
\end{proposition}

\begin{proof}
Let \(d\in(a\boxplus b)\boxplus c\). By~\eqref{fn}, the sets \(X:=(-a)\boxplus d\) and \(Y:=b\boxplus c\) are nested. Totality gives \(X\neq\emptyset\neq Y\). Nested nonempty sets have nonempty intersection, so \(X\cap Y\neq\emptyset\). Lemma~\ref{lem:triple-inter} yields \(d\in a\boxplus(b\boxplus c)\). The dual inclusion is symmetric.
\end{proof}

\begin{proposition}[Associativity implies factor nesting under product-ultrametricity]
\label{prop:asc-implies-fn}
Let \((A,\boxplus,0)\) be a total product-ultrametric mosaic. If \(A\) is associative, then it satisfies factor nesting.
\end{proposition}

\begin{proof}
Suppose \(d\in(a\boxplus b)\boxplus c\). Associativity gives \(d\in a\boxplus(b\boxplus c)\), so Lemma~\ref{lem:triple-inter} yields \(((-a)\boxplus d)\cap(b\boxplus c)\neq\emptyset\). Product-ultrametricity implies that \((-a)\boxplus d\) and \(b\boxplus c\) are nested. The other clause is symmetric.
\end{proof}

\begin{theorem}[Associativity via factor nesting]
\label{thm:asc-fn}
Let \((A,\boxplus,0)\) be a total product-ultrametric mosaic. The following are equivalent:
\begin{enumerate}[label=(\roman*)]
\item\label{asc-i} \(A\) is associative;
\item\label{fn-i} \(A\) satisfies factor nesting~\eqref{fn}.
\end{enumerate}
In particular this applies to \(\mathbb{K}\) and to \(\T(\Gamma)\), viewed as total product-ultrametric valued mosaics.
\end{theorem}

\begin{proof}
Combine Propositions~\ref{prop:fn-implies-asc} and~\ref{prop:asc-implies-fn}.
\end{proof}

\begin{remark}
\label{rem:fn-proof-note}
Totality is essential for Proposition~\ref{prop:fn-implies-asc}: without it, factor nesting need not yield associativity. Free mosaics satisfy~\eqref{fn} and fail both totality and associativity (Corollary~\ref{cor:fn-not-asc}). Product-ultrametricity is essential for Proposition~\ref{prop:asc-implies-fn}: associativity alone does not force nesting of binary products (see e.g.\ phase-type polygroups in the literature).
\end{remark}

\begin{example}
\label{ex:fn-examples}
\begin{enumerate}[label=(\alph*)]
\item \(\T(\Gamma)\) and \(\mathbb{K}\) are total and product-ultrametric; by Theorem~\ref{thm:asc-fn} they satisfy factor nesting. 
\item \(\mathbf{F}(S)\) is product-ultrametric and satisfies factor nesting, but is neither total nor associative (Corollary~\ref{cor:fn-not-asc}).
\item \(\mathbb{K}\times\mathbb{K}\) is total and associative yet fails both product-ultrametricity and factor nesting (Proposition~\ref{prop:asc-not-fn}).
\end{enumerate}
\end{example}

\subsection{Krasner valuations: from values to balls}

The weak axioms~\eqref{v1}--\eqref{v4} constrain the \emph{values} of elements of sums. They do not force sums, as subsets, to nest. Factor nesting organises the combinatorial shadow of associativity: under totality it yields ASC (Proposition~\ref{prop:fn-implies-asc}), and under totality and product-ultrametricity it is equivalent to ASC (Theorem~\ref{thm:asc-fn}). What is still missing is a geometric upgrade of the valuation itself, i.e., one that forces sums to be ultrametric balls. Combined with totality and factor nesting, that upgrade produces the superiorly canonical package (Theorem~\ref{thm:kvh-sch}).

That upgrade is the additive content of Krasner's valuation axiom from~\cite{Lin23}, isolated here without multiplication. Write \(a\boxminus b:=a\boxplus(-b)\) for the multivalued difference.

\begin{definition}[Initial segment]
\label{def:initial}
A subset \(\rho\subseteq\Gamma\) is an \emph{initial segment} if \(\delta\in\rho\) and \(\gamma<\delta\) imply \(\gamma\in\rho\). Write \(\gamma>\rho\) to mean \(\gamma>\delta\) for all \(\delta\in\rho\) (equivalently \(\gamma\notin\rho\) when \(\rho\) is initial).
\end{definition}

\begin{definition}[Additive Krasner valuation]
\label{def:kvh}
A valued mosaic \((A,v)\) over \(\Gamma\) is a \emph{Krasner valued mosaic}, and \(v\) an \emph{additive Krasner valuation}, if there exists an initial segment \(\rho_v\subseteq\Gamma\) with \(0\in\rho_v\) (the \emph{norm} of \(v\)) such that for all \(x,y,z,t\in A\) with \(z\in x\boxplus y\),
\begin{enumerate}[label=(KVH)]
\item\label{kvh}
\(t\in x\boxplus y\) if and only if \(v(s)>\rho_v+\min\bigl(v(x),v(y)\bigr)\) for every \(s\in z\boxminus t\).
\end{enumerate}
\end{definition}

Informally,~\eqref{kvh} says that once a sum \(x\boxplus y\) is inhabited, membership in that sum is completely determined by an open-ball condition in the value group, measured from any centre \(z\in x\boxplus y\). Thus~\eqref{kvh} upgrades a weak valuation from a constraint on values to a description of sums as sets.

\begin{lemma}[Unique difference values and balls]
\label{lem:unique-diff}
Let \((A,v)\) be a total Krasner valued mosaic that satisfies factor nesting~\eqref{fn}.
Then \(A\) is associative (hence a valued polygroup), and:
\begin{enumerate}[label=(\roman*)]
\item\label{ud-i} if \(x\neq y\), then \(v\) is constant on \(x\boxminus y\); write \(d_v(x,y)\) for that common value and set \(d_v(x,x)=\infty\);
\item\label{ud-ii} \(d_v\) is an ultrametric on \(A\) with values in \(\Gamma_\infty\);
\item\label{ud-iii} for all \(x,y\in A\) and every \(z\in x\boxplus y\),
\[
x\boxplus y
=\bigl\{t\in A \bigm| d_v(z,t)>\rho_v+\min\bigl(v(x),v(y)\bigr)\bigr\}.
\]
\end{enumerate}
\end{lemma}

\begin{proof}
By Proposition~\ref{prop:fn-implies-asc}, the total mosaic \(A\) is associative, so \((A,v)\) is a valued polygroup.
We follow~\cite[Proposition~4.8]{Lin23}, using totality, associativity,~\eqref{v1}--\eqref{v4}, and~\eqref{kvh}.

\eqref{ud-i}. Let \(x\neq y\) and \(z,t\in x\boxminus y\), and suppose \(v(z)<v(t)\). Then necessarily \(v(x)=v(y)\): otherwise~\eqref{v4} would force every element of \(x\boxminus y=x\boxplus(-y)\) to have value \(\min(v(x),v(y))\). Since \(z\in x\boxminus y\) and \(t\in x\boxminus y\), the forward implication in~\eqref{kvh} (with product \(x\boxminus y\) and centre \(z\)) yields \(v(s)>\rho_v+\min(v(x),v(y))\) for all \(s\in z\boxminus t\). But~\eqref{v4} and \(v(z)<v(t)\) give \(v(s)=v(z)\) for such \(s\), so \(v(z)>\rho_v+\min(v(x),v(y))\). Applying~\eqref{kvh} again to test membership of \(0\) in \(x\boxminus y\): the condition is \(v(s)>\rho_v+\min(v(x),v(y))\) for all \(s\in z\boxminus 0=\{z\}\), which holds. Hence \(0\in x\boxminus y\), so \(x=y\), a contradiction. Thus \(v(z)=v(t)\).

\eqref{ud-ii}. If \(x=y\), then \(d_v(x,y)=\infty\) by definition. Conversely, if \(x\neq y\), then~\eqref{ud-i} supplies a constant value \(\gamma\) of \(v\) on the nonempty set \(x\boxminus y\); this \(\gamma\) cannot be \(\infty\), for otherwise every \(s\in x\boxminus y\) would satisfy \(v(s)=\infty\), hence \(s=0\) by~\eqref{v1}, so \(x\boxminus y=\{0\}\) and \(0\in x\boxminus y\), whence \(x=y\). Thus \(d_v(x,y)=\infty\Leftrightarrow x=y\). Symmetry uses \(v(-a)=v(a)\).

For the ultrametric inequality \(d_v(x,z)\ge\min\bigl(d_v(x,y),d_v(y,z)\bigr)\), the cases where any two of \(x,y,z\) coincide are immediate. Assume they are pairwise distinct, and fix \(r\in x\boxminus z\), so \(d_v(x,z)=v(r)\). Since \(0\in(-y)\boxplus y\) and the mosaic is associative,
\begin{align*}
x\boxplus(-z)
&=x\boxplus\bigl(0\boxplus(-z)\bigr)\\
&\subseteq x\boxplus\bigl(\bigl((-y)\boxplus y\bigr)\boxplus(-z)\bigr)\\
&=x\boxplus\bigl((-y)\boxplus\bigl(y\boxplus(-z)\bigr)\bigr)\\
&=\bigl(x\boxplus(-y)\bigr)\boxplus\bigl(y\boxplus(-z)\bigr).
\end{align*}
Thus \(r\in p\boxplus q\) for some \(p\in x\boxminus y\) and \(q\in y\boxminus z\). Axiom~\eqref{v3} yields \(v(r)\ge\min\bigl(v(p),v(q)\bigr)=\min\bigl(d_v(x,y),d_v(y,z)\bigr)\), as required.

\eqref{ud-iii}. Fix \(z\in x\boxplus y\). By~\eqref{kvh}, \(t\in x\boxplus y\) if and only if \(v(s)>\rho_v+\min(v(x),v(y))\) for all \(s\in z\boxminus t\). For \(t\neq z\), the latter means \(d_v(z,t)>\rho_v+\min(v(x),v(y))\) by~\eqref{ud-i}; for \(t=z\), both sides hold.
\end{proof}

\begin{remark}[Balls nest]
\label{rem:kvh-total}
Intersecting open balls in an ultrametric space are nested~\cite[Lemma~4.5(iii)]{Lin23}. Combined with Lemma~\ref{lem:unique-diff}\eqref{ud-iii}, this yields product-ultrametricity under the hypotheses of that lemma (total\({}+{}\)FN\({}+{}\)KVH):~\eqref{kvh} forces sums to be balls once associativity is available, and balls nest.
\end{remark}

\begin{definition}[Superiorly canonical mosaic]
\label{def:sch}
A commutative mosaic \(A\) is \emph{superiorly canonical} if it is total and satisfies:
\begin{enumerate}[label=(SCH\arabic*)]
\item\label{sch1} if \(x\in x\boxplus y\), then \(x\boxplus y=\{x\}\);
\item\label{sch2} if \((x\boxplus y)\cap(z\boxplus t)\neq\emptyset\), then \(x\boxplus y\) and \(z\boxplus t\) are nested (product-ultrametricity);
\item\label{sch3} if \(x\neq y\) and \(z,t\in x\boxminus y\), then \(z\boxminus z=t\boxminus t\);
\item\label{sch4} if \(x\in z\boxminus z\) and \(y\notin z\boxminus z\), then \(x\boxminus x\subseteq y\boxminus y\).
\end{enumerate}
\end{definition}

(This is the additive content of Mittas' superiorly canonical hypergroups as in~\cite[Definition~4.17]{Lin23}.)

\begin{theorem}[Factor nesting and Krasner balls imply superior canonicity]
\label{thm:kvh-sch}
Let \((A,v)\) be a total Krasner valued mosaic that satisfies factor nesting~\eqref{fn}.
Then \(A\) is associative (hence a valued polygroup) and superiorly canonical.
In particular multioperation outputs are open ultrametric balls and \(A\) is product-ultrametric.
\end{theorem}

By Proposition~\ref{prop:fn-implies-asc}, the same conclusions hold if one assumes associativity in place of factor nesting.

\begin{proof}
Lemma~\ref{lem:unique-diff} supplies associativity and identifies inhabited products with open balls.
We adapt~\cite[Proposition~4.19]{Lin23}. Write \(\rho=\rho_v\).

\eqref{sch1}. Suppose \(x\in x\boxplus y\). Reversibility gives \(y\in x\boxminus x\). Since \(0\in x\boxminus x\) (Lemma~\ref{lem:inv-unit}), Lemma~\ref{lem:unique-diff}\eqref{ud-iii} shows that \(x\boxminus x\) is the open ball of centre \(0\) and radius \(\rho+v(x)\), so \(v(y)>\rho+v(x)\ge v(x)\). If \(z\in x\boxplus y\), then~\eqref{v4} and \(v(y)>v(x)\) force \(v(z)=v(x)\). Reversibility yields \(y\in z\boxminus x\). Then \(v(y)>\rho+v(x)=\rho+\min(v(z),v(x))\), so Lemma~\ref{lem:unique-diff}\eqref{ud-iii} (equivalently~\eqref{kvh}) gives \(0\in z\boxminus x\), hence \(z=x\). Thus \(x\boxplus y=\{x\}\).

\eqref{sch2}. Products are open ultrametric balls (Lemma~\ref{lem:unique-diff}\eqref{ud-iii}); intersecting open balls are nested~\cite[Lemma~4.5(iii)]{Lin23}.

\eqref{sch3}. Let \(x\neq y\) and \(z,t\in x\boxminus y\). Lemma~\ref{lem:unique-diff}\eqref{ud-i} gives \(v(z)=v(t)\). By~\eqref{kvh}, \(a\in z\boxminus z\) if and only if \(v(a)>\rho+v(z)\), and likewise for \(t\boxminus t\). Hence \(z\boxminus z=t\boxminus t\).

\eqref{sch4}. If \(x\in z\boxminus z\) and \(y\notin z\boxminus z\), then~\eqref{kvh} gives \(v(x)>\rho+v(z)\) and \(v(y)\not>\rho+v(z)\), so \(v(x)>v(y)\). For \(a\in x\boxminus x\),~\eqref{kvh} yields \(v(a)>\rho+v(x)>\rho+v(y)\), hence \(a\in y\boxminus y\).
\end{proof}

\begin{remark}[The dual direction]
\label{rem:sch-to-kvh}
On a hyperfield, superior canonicity of the additive structure implies the existence of a Krasner valuation~\cite[Proposition~4.20]{Lin23}: the multiplicative monoid is used to build the valuation hyperring \(\{x\mid x\boxminus x\subseteq 1\boxminus 1\}\) and to recover~\eqref{kvh}. We do not claim an additive-only converse for general mosaics; isolating such a reconstruction (or showing that multiplication is essential) is left open.
\end{remark}

\begin{corollary}[Hierarchy]
\label{cor:hierarchy}
The layers of structure may be summarised as follows.
\begin{enumerate}[label=(\alph*)]
\item\label{hier-weak} \emph{Weak valuation}~\eqref{v1}--\eqref{v4}: ambient objects of \(\VMscG\) and unit-reflecting morphisms to \(\T(\Gamma)\). Values of summands are constrained; nesting of sums as sets is not (Remark~\ref{rem:val-not-ulm}).
\item\label{hier-fn} \emph{Factor nesting}~\eqref{fn}: a combinatorial shadow of associativity. Without totality it does not imply ASC (free mosaics); ASC alone does not imply FN (\(\mathbb{K}\times\mathbb{K}\)).
\item\label{hier-char} \emph{Characterisation:} total\({}+{}\)FN\(\Rightarrow\)ASC (Proposition~\ref{prop:fn-implies-asc}); for total product-ultrametric mosaics, \(\mathrm{ASC}\Leftrightarrow\mathrm{FN}\) (Theorem~\ref{thm:asc-fn}).
\item\label{hier-kvh} \emph{Krasner axiom}~\eqref{kvh}: total\({}+{}\)FN\({}+{}\)KVH\(\Rightarrow\)ASC and the superiorly canonical package (Theorem~\ref{thm:kvh-sch}); product-ultrametricity is then geometric (sums are balls).
\end{enumerate}
In short:
\[
\text{FN}\;\not\Rightarrow\;\text{ASC},
\qquad
\text{ASC}\;\not\Rightarrow\;\text{FN},
\qquad
\text{product-ultrametric}+\text{FN}\;\not\Rightarrow\;\text{total},
\]
while
\[
\text{total}+\text{FN}\;\Rightarrow\;\text{ASC},
\qquad
\text{total}+\text{product-ultrametric}\;\Rightarrow\;(\mathrm{ASC}\Leftrightarrow\mathrm{FN}),
\]
and
\[
\text{total}+\text{FN}+\text{KVH}\;\Rightarrow\;\text{ASC}+\text{SCH}.
\]
\end{corollary}

\section{Outlook}
\label{sec:outlook}

The present paper supplies a valuation-first ambient category---valuations as unit-reflecting morphisms to \(\T(\Gamma)\), a complete and cocomplete strict slice, a finitely complete lax category with all small coproducts---and analyses associativity inside that world via factor nesting (Proposition~\ref{prop:fn-implies-asc}, Theorem~\ref{thm:asc-fn}), with the Krasner ball axiom as a geometric upgrade from total\({}+{}\)FN to the superiorly canonical package (Theorem~\ref{thm:kvh-sch}). Natural continuations include the following.

\begin{enumerate}[label=(\arabic*)]
\item \textbf{Lax coequalizers for nontrivial \(\Gamma\).} Determine whether short coequalizers of~\cite{NR23} admit a universal valuation for inequality morphisms in \(\VMscG\) (the strict slice is already cocomplete by Corollary~\ref{cor:strict-slice-coeq}).
\item \textbf{Lax vs.\ strict comparison.} Develop further the relationship between \(\VMscG\) and \(\cMsc/\T(\Gamma)\), including when the \(\min\)-product and the pullback product interact with free constructions and with \(\mathbf{VPol}_\Gamma\).
\item \textbf{Additive reconstruction of Krasner valuations.} Decide whether a purely additive superiorly canonical (or FN+SCH) package on a total mosaic can induce a valuation satisfying~\eqref{kvh}, or whether multiplicative structure as in~\cite[Proposition~4.20]{Lin23} is essential (Remark~\ref{rem:sch-to-kvh}).
\item \textbf{Completion in the slice.} Extend the slice-limit description of Krasner completion~\cite{Lin23compl} from valued hyperfields to diagrams in \(\cMsc/\T(\Gamma)\) or \(\VMscG\).
\item \textbf{Monoidal and multiplicative structure.} Lift the closed monoidal structure of \(\cMsc\)~\cite{NR23} to the valued setting, and define ring-/field-like objects connecting to Krasner valued hyperfields~\cite{Lin23,PhD22}.
\item \textbf{Morphisms.} Characterise monomorphisms, epimorphisms, and isometric embeddings in \(\VMscG\) and in \(\cMsc/\T(\Gamma)\), in the spirit of~\cite{NR23}.
\end{enumerate}

\section*{Acknowledgements}

Parts of the preparation of this manuscript were assisted by the AI system Grok~4.5 (xAI), used via the Grok Build CLI, for drafting, editing, and checking mathematical arguments.
The author retains full responsibility for the correctness of all results and for the final text.

\bibliographystyle{abbrv}
\bibliography{Biblio-compl}

\end{document}